\documentclass{amsart}
\usepackage{amscd,amsmath,amssymb}
\usepackage{pstricks}
\usepackage{latexsym,amsbsy,mathrsfs}
\usepackage{xy}
\usepackage{xypic}
\usepackage{pgf,tikz}

\newtheorem{thm}{Theorem}[section]
\newtheorem{lem}[thm]{Lemma}
\newtheorem{cor}[thm]{Corollary}
\newtheorem{prop}[thm]{Proposition}

\theoremstyle{definition}
\newtheorem{example}[thm]{Example}

\newtheorem{defn}[thm]{Definition}

\newtheorem{rem}[thm]{Remark}

\numberwithin{equation}{thm}

\begin{document}
\title[abelian quotients  via morphism categories]
{abelian quotients arising from extriangulated categories via morphism categories}

\author{Zengqiang Lin}
\address{ School of Mathematical sciences, Huaqiao University,
Quanzhou\quad 362021,  China.} \email{lzq134@163.com}

\thanks{This work was supported by the national natural science foundation of China (Grants No. 11871014 and No. 11871259)}



\begin{abstract}
We investigate abelian quotients arising from extriangulated categories via morphism categories, which is
  a unified treatment for both exact categories and triangulated categories. Let $(\mathcal{C},\mathbb{E},\mathfrak{s})$ be an extriangulated category with enough projectives $\mathcal{P}$ and $\mathcal{M}$ be a full subcategory of $\mathcal{C}$ containing $\mathcal{P}$. We show that certain quotient category of $\mathfrak{s}\textup{-def}(\mathcal{M})$, the category of $\mathfrak{s}$-deflations $f:M_{1}\rightarrow M_2$ with $M_1,M_2\in\mathcal{M}$, is abelian. Our main theorem has two applications. If $\mathcal{M}=\mathcal{C}$, we obtain that certain ideal quotient category $\mathfrak{s}\textup{-tri}(\mathcal{C})/\mathcal{R}_2$ is equivalent to the category of finitely presented modules $\textup{mod-}\mathcal{C}/[\mathcal{P}]$, where $\mathfrak{s}$-tri$(\mathcal{C})$ is the category of all $\mathfrak{s}$-triangles. If $\mathcal{M}$ is a rigid subcategory, we show that  $\mathcal{M}_{L}/[\mathcal{M}]\cong\textup{mod-}(\mathcal{M}/[\mathcal{P}])$ and $\mathcal{M}_{L}/[\Omega\mathcal{M}]\cong(\textup{mod-}(\mathcal{M}/[\mathcal{P}])^{\textup{op}})^{\textup{op}}$, where $\mathcal{M}_L$ (resp. $\Omega\mathcal{M}$) is the full subcategory of $\mathcal{C}$  of objects $X$ admitting an $\mathfrak{s}$-triangle $\xymatrixrowsep{0.1pc}\xymatrix{X\ar[r]&M_1\ar[r] & M_2\ar@{-->}[r]&} (\textup{resp.} \xymatrixrowsep{0.1pc}\xymatrix{X\ar[r]&P\ar[r] & M\ar@{-->}[r]&})$
   with $M_1, M_2\in\mathcal{M}$ (resp. $M\in\mathcal{M}$ and $P\in\mathcal{P}$).
   In particular, we  have $\mathcal{C}/[\mathcal{M}]\cong\textup{mod-}(\mathcal{M}/[\mathcal{P}])$ and $\mathcal{C}/[\Omega\mathcal{M}]\cong(\textup{mod-}(\mathcal{M}/[\mathcal{P}])^{\textup{op}})^{\textup{op}}$ provided that $\mathcal{M}$ is a cluster-tilting subcategory.
\end{abstract}

\maketitle

\section{introduction}

%It is an interesting question which quotient categories admit abelian structures.
In representation theory, there are a few quotient categories admitting natural abelian structures.
For both triangulated categories and exact categories, cluster-tilting subcategories  provide a way to construct abelian quotient categories.
%Cluster-tilting theory provides a way to construct abelian quotient categories.
Let $\mathcal{C}$ be a triangulated category and $\mathcal{T}$ be a cluster-tilting subcategory of $\mathcal{C}$, then the quotient $\mathcal{C}/[\mathcal{T}]$ is abelian; related works see \cite{[BMR],[KR],[KZ]}. The version of exact categories see \cite{[DL]}.
Submodule categories provide another way to construct abelian quotient categories. Certain quotients of submodule categories are realized as categories of finitely presented modules over stable Auslander algebras \cite{[RZ],[Eir]}. More generally, some quotients of  categories of short exact sequences in exact categories are abelian, related works see \cite{[Gen],[Lin]}. For triangulated version, certain quotients of categories of triangles are abelian \cite{[Nee]}.

Recently, Nakaoka and Palu introduced the notion of extriangulated categories \cite{[NP]}, which is  a simultaneous generalization of exact categories and triangulated categories. They pointed out that the notion is a convenient setup for writing down proofs which apply to both exact categories and triangulated categories. For recent developments on extriangulated categories we refer to \cite{[INP],[LN],[LZ1],[LZ2],[ZZ]} etc.

 In this paper, we focus our attention onto the abelian quotients arising from extriangulated categories via morphism categories, which is
  a unified treatment of abelian quotients for both exact categories and triangulated categories. Our approach to abelian quotients is based on identifying quotients of morphism categories as certain module categories.

Let $(\mathcal{C},\mathbb{E},\mathfrak{s})$ be an extriangulated category and $\mathcal{M}$ be a full subcategory of $\mathcal{C}$. We denote by $\textup{Mor}(\mathcal{M})$ the morphism category of $\mathcal{M}$ and by $\mathfrak{s}$-def$(\mathcal{M})$ (resp. $\mathfrak{s}\textup{-inf}(\mathcal{M})$) the full subcategory of $\textup{Mor}(\mathcal{M})$ consisting of $\mathfrak{s}$-deflations (resp. $\mathfrak{s}$-inflations). The full subcategory of $\mathfrak{s}$-def$(\mathcal{M})$ consisting of split epimorphisms (resp. split monomorphisms) is denoted by s-epi$(\mathcal{M})$ (resp. s-mono$(\mathcal{M})$).  We denote by sp-epi$(\mathcal{M})$ (resp. si-mono$(\mathcal{M})$) the full subcategory of $\mathfrak{s}$-def$(\mathcal{M})$ consisting of $(M\xrightarrow{1}M)\oplus(P\rightarrow M')$ (resp. $(M\xrightarrow{1}M)\oplus(M'\rightarrow I)$) with $P\in\mathcal{P}$ (resp. $I\in\mathcal{I}$).

Our main theorem is the following (Theorem \ref{3.1}), which generalizes \cite[Theorem 3.9]{[Lin]}.

\begin{thm}\label{thm1.1}
Let $\mathcal{C}$ be an extriangulated category and $\mathcal{M}$ be a full subcategory of $\mathcal{C}$.

\textup{(1)} If $\mathcal{C}$ has enough projectives $\mathcal{P}$ and $\mathcal{M}$ contains $\mathcal{P}$, then   $\mathfrak{s}\textup{-def}(\mathcal{M})/[\textup{s-epi}(\mathcal{M})]\cong \textup{mod-}\mathcal{M}/[\mathcal{P}]$ and  $\mathfrak{s}\textup{-def}(\mathcal{M})/[\textup{sp-epi}(\mathcal{M})]\cong (\textup{mod-}(\mathcal{M}/[\mathcal{P}])^{\textup{op}})^{\textup{op}}$.

\textup{(2)} If $\mathcal{C}$ has enough injectives $\mathcal{I}$ and $\mathcal{M}$ contains $\mathcal{I}$, then
$\mathfrak{s}\textup{-inf}(\mathcal{M})/[\textup{s-mono}(\mathcal{M})]\cong (\textup{mod-}(\mathcal{M}/[\mathcal{I}])^{\textup{op}})^{\textup{op}}$ and  $\mathfrak{s}\textup{-inf}(\mathcal{M})/[\textup{si-mono}(\mathcal{M})]\cong \textup{mod-}\mathcal{M}/[\mathcal{I}]$.
%In particular, we have $\mathfrak{s}\textup{-def}(\mathcal{C})/[\textup{s-epi}(\mathcal{C})]\cong \textup{mod-}\mathcal{C}/[\mathcal{P}].$
\end{thm}

Theorem \ref{thm1.1} has two interesting applications. We will investigate two special cases when $\mathcal{M}=\mathcal{C}$ and when $\mathcal{M}$ is {\em rigid}, that is, $\mathbb{E}(M,M')=0$ for any $M,M'\in\mathcal{M}$.

For the first case, we denote by $\mathfrak{s}$-tri$(\mathcal{C})$ the category of all $\mathfrak{s}$-triangles, where the objects are the $\mathfrak{s}$-triangles $X_\bullet=(\xymatrix{X_1\ar[r]^{f_1}&X_2\ar[r]^{f_2} & X_3\ar@{-->}[r]^{\delta}&})$ and the morphisms from $X_\bullet$ to $Y_\bullet$ are the triples $\varphi_\bullet=(\varphi_1,\varphi_2,\varphi_3)$ such that the following diagram
$$\xymatrix{
X_1 \ar[r]^{f_1} \ar[d]^{\varphi_1} & X_2 \ar[r]^{f_2}\ar[d]^{\varphi_2} & X_3 \ar@{-->}[r]^{\delta}\ar[d]^{\varphi_3} & \\
Y_1\ar[r]^{g_1} & Y_2\ar[r]^{g_2} & Y_3\ar@{-->}[r]^{\delta'} &
}$$
is a morphism of $\mathfrak{s}$-triangles.
 Let $X_\bullet$ and $Y_\bullet$ be two $\mathfrak{s}$-triangles, we denote by $\mathcal{R}_2(X_\bullet,Y_\bullet)$
  the class of morphisms $\varphi_\bullet: X_\bullet\rightarrow Y_\bullet$ such that $\varphi_3$ factors through $g_2$. %there is a morphism $p:X_3\rightarrow Y_2$ such that $g_2p=\varphi_3$.
   It is easy to see that $\mathcal{R}_2$ is an ideal of $\mathfrak{s}\textup{-tri}(\mathcal{C})$, moreover, the following three quotient categories $\mathfrak{s}\textup{-tri}(\mathcal{C})/\mathcal{R}_2,\mathfrak{s}\textup{-def}(\mathcal{C})/[\textup{s-epi}(\mathcal{C})]$ and $\mathfrak{s}\textup{-inf}(\mathcal{C})/[\textup{s-mono}(\mathcal{C})]$ are equivalent.

 Given an $\mathfrak{s}$-triangle  $\delta=(\xymatrix{X_1\ar[r]^{f_1}&X_2\ar[r]^{f_2} & X_3\ar@{-->}[r]^{\rho}&})$, we define  the {\em contravariant defect} $\delta^\ast$  and the {\em covariant defect} $\delta_\ast$ by the following exact sequence of functors
$$\mathcal{C}(-,X_{1})\xrightarrow{\mathcal{C}(-,f_{1})}\mathcal{C}(-,X_{2})\xrightarrow{\mathcal{C}(-,f_{2})}\mathcal{C}(-,X_3)\rightarrow \delta^\ast\rightarrow 0,$$
$$\mathcal{C}(X_{3},-)\xrightarrow{\mathcal{C}(f_{2},-)}\mathcal{C}(X_{2},-)\xrightarrow{\mathcal{C}(f_{1},-)}\mathcal{C}(X_1,-)\rightarrow \delta_\ast\rightarrow 0.$$

Our first application of Theorem \ref{thm1.1} is the following (Theorem \ref{thm3.2}, Proposition \ref{prop1} and Theorem \ref{rem3.2}), which generalizes \cite[Theorem 4.1, Theorem 4.8, Theorem 5.1]{[Lin]}.

\begin{thm}
Let $\mathcal{C}$ be an extriangulated category.

\textup{(1)} The quotient $\mathfrak{s}$-tri$(\mathcal{C})/\mathcal{R}_2$ is abelian.

\textup{(2)} If $\mathcal{C}$ has enough projectives $\mathcal{P}$, then we have the following equivalence $$F:\mathfrak{s}\textup{-tri}(\mathcal{C})/\mathcal{R}_2\cong \textup{mod-}\mathcal{C}/[\mathcal{P}],\delta\mapsto \delta^*.$$

\textup{(3)} If $\mathcal{C}$ has enough injectives $\mathcal{I}$, then we have the following equivalence $$G:\mathfrak{s}\textup{-tri}(\mathcal{C})/\mathcal{R}_2\cong (\textup{mod-}(\mathcal{C}/[\mathcal{I}])^{\textup{op}})^{\textup{op}},\delta\mapsto\delta_*.$$
\end{thm}

We point out that the abelian quotient $\mathfrak{s}$-tri$(\mathcal{C})/\mathcal{R}_2$ admits nice properties. We describe the projectives, injectives and simple objects in $\mathfrak{s}$-tri$(\mathcal{C})/\mathcal{R}_2$; see Proposition \ref{prop3.2} and Proposition \ref{prop3.2.0}. In particular, if $\mathcal{C}$ has enough projectives $\mathcal{P}$ and enough injectives  $\mathcal{I}$, then there is a duality between mod-$\mathcal{C}/[\mathcal{P}]$ and mod-$(\mathcal{C}/[\mathcal{I}])^{\textup{op}}$, which is used to derive Auslander-Reiten duality and defect formula for extriangulated categories; see Proposition \ref{prop3.2.1}.

For describing  the second application, we first give some notations.
Let  $\mathcal{C}'$ and $\mathcal{C}''$ be two full subcategories of $\mathcal{C}$. We denote by Cocone$(\mathcal{C}',\mathcal{C}'')$ (resp. Cone$(\mathcal{C}',\mathcal{C}'')$) the subcategory of objects $X$ admitting an $\mathfrak{s}$-triangle $\xymatrixrowsep{0.1pc}\xymatrix{X\ar[r]&C'\ar[r] & C''\ar@{-->}[r]&}$ (resp. $\xymatrixrowsep{0.1pc}\xymatrix{C'\ar[r]&C''\ar[r] & X\ar@{-->}[r]&}$) with $C'\in\mathcal{C}'$ and $C''\in\mathcal{C}''$. %We denote by Cone$(\mathcal{C}',\mathcal{C}'')$ the subcategory of objects $X$ admitting an $\mathfrak{s}$-triangle $\xymatrixrowsep{0.1pc}\xymatrix{C'\ar[r]&C''\ar[r] & X\ar@{-->}[r]&}$ with $C'\in\mathcal{C}'$ and $C''\in\mathcal{C}''$.

Let $\mathcal{M}$ be a rigid subcategory of $\mathcal{C}$.  For convenience we let $\mathcal{M}_L=\textup{Cocone}(\mathcal{M},\mathcal{M})$ and $\mathcal{M}_R=\textup{Cone}(\mathcal{M},\mathcal{M})$. If $\mathcal{C}$ has enough projectives $\mathcal{P}$ and enough injectives $\mathcal{I}$, then we let $\Omega\mathcal{M}$=Cocone$(\mathcal{P},\mathcal{M})$ and $\Sigma\mathcal{M}$=Cone$(\mathcal{M},\mathcal{I})$. It turns out that the quotient categories  $\mathfrak{s}\textup{-def}(\mathcal{M})/[\textup{s-epi}(\mathcal{M})]$ and $\mathfrak{s}\textup{-def}(\mathcal{M})/[\textup{sp-epi}(\mathcal{M})]$ can be realized as subquotient categories of $\mathcal{C}$. %We say $\mathcal{M}$ is {\em rigid} if $\mathbb{E}(M,M')=0$ for each objects $M,M'\in\mathcal{M}$.

Our second application of Theorem \ref{thm1.1} is the following (Theorem \ref{thm4.1}), which generalizes \cite[Theorem 3.2, Theorem 3.4]{[DL]} and \cite[Proposition 6.2]{[IY]}.

\begin{thm}
Let $\mathcal{C}$ be an extriangulated category and $\mathcal{M}$ be a rigid subcategory of $\mathcal{C}$.

\textup{(1)} If $\mathcal{C}$ has enough projectives $\mathcal{P}$ and $\mathcal{M}$ contains $\mathcal{P}$, then $\mathcal{M}_{L}/[\mathcal{M}]\cong\textup{mod-}(\mathcal{M}/[\mathcal{P}])$ and  $\mathcal{M}_{L}/[\Omega\mathcal{M}]\cong(\textup{mod-}(\mathcal{M}/[\mathcal{P}])^{\textup{op}})^{\textup{op}}$.

\textup{(2)} If $\mathcal{C}$ has enough injectives $\mathcal{I}$ and $\mathcal{M}$  contains $\mathcal{I}$, then $\mathcal{M}_{R}/[\mathcal{M}]\cong(\textup{mod-}(\mathcal{M}/[\mathcal{I}])^{\textup{op}})^{\textup{op}}$ and  $\mathcal{M}_{R}/[\Sigma\mathcal{M}]\cong\textup{mod-}(\mathcal{M}/[\mathcal{I}])$.
\end{thm}

In particular, if $\mathcal{M}$ is a cluster tilting subcategory of $\mathcal{C}$, then $\mathcal{M}_L=\mathcal{M}_R=\mathcal{C}$. Thus we have the following result (Corollary \ref{cor5.1}).

\begin{cor}
Let $\mathcal{C}$ be an extriangulated category with enough projectives $\mathcal{P}$ and enough injectives $\mathcal{I}$. If $\mathcal{M}$ is a cluster tilting subcategory of $\mathcal{C}$, then

\textup{(1)} $\mathcal{C}/[\mathcal{M}]\cong\textup{mod-}(\mathcal{M}/[\mathcal{P}])\cong(\textup{mod-}(\mathcal{M}/[\mathcal{I}])^{\textup{op}})^{\textup{op}}$.

\textup{(2)} $\mathcal{C}/[\Omega\mathcal{M}]\cong(\textup{mod-}(\mathcal{M}/[\mathcal{P}])^{\textup{op}})^{\textup{op}}$.

\textup{(3)} $\mathcal{C}/[\Sigma\mathcal{M}]\cong\textup{mod-}(\mathcal{M}/[\mathcal{I}])$.
\end{cor}

This  paper is organized as follows. In Section 2 we make some preliminaries on morphism categories and extriangulated categories. In Section 3 we prove  Theorem 1.1. In Section 4 we provide the first application. In Section 5 we provide the second application.

\section{definitions and preliminaries}

In this section, we first give some facts on morphism categories, then recall the definitions and basic properties on extriangulated categories from \cite{[NP]}, \cite{[LN]} and \cite{[INP]}.

\subsection{Morphism categories}

Let $\mathcal{C}$ be an additive category. The {\em morphism category} of $\mathcal{C}$ is the category $\textup{Mor}(\mathcal{C})$ defined by the following data.
The objects of $\textup{Mor}(\mathcal{C})$ are all  the morphisms $f:X\rightarrow Y$ in $\mathcal{C}$. The morphisms from $f:X\rightarrow Y$ to $f':X'\rightarrow Y'$ are pairs $(a,b)$ where $a:X\rightarrow X'$ and $b:Y\rightarrow Y'$ such that $bf=f'a$. The composition of morphisms is componentwise. For two objects $f:X\rightarrow Y$ and $f':X'\rightarrow Y'$ in $\textup{Mor}(\mathcal{C})$, we define $\mathcal{R}(f,f')$ (resp. $\mathcal{R'}(f,f'))$ to be the set of morphisms $(a,b)$ such that there is some morphism $p:Y\rightarrow X'$ such that $f'p=b$ (resp. $pf=a$). Then $\mathcal{R}$ and $\mathcal{R}'$ are ideals of $\textup{Mor}(\mathcal{C})$. We denote by s-epi$(\mathcal{C})$ (resp. s-mono$(\mathcal{C})$) the full subcategory of $\textup{Mor}(\mathcal{C})$ consisting of split epimorphisms (resp. split monomorphisms).

Recall that a right {\em $\mathcal{C}$-module} is a contravariantly additive functor $F:\mathcal{C}\rightarrow Ab$, where $Ab$ is the category of abelian groups. A $\mathcal{C}$-module $F$ is called {\em{finitely presented}} if there exists an exact sequence $\mathcal{C}(-,X)\rightarrow\mathcal{C}(-,Y)\rightarrow F\rightarrow 0$. We denote by mod-$\mathcal{C}$ the category of finitely presented $\mathcal{C}$-modules, and by $\textup{proj-}\mathcal{C}$ (resp. $\textup{inj-}\mathcal{C}$) the full subcategory of mod-$\mathcal{C}$ consisting of projectives (resp. injectives).

The following result was proved in \cite[Lemma 3.1]{[Lin]} and \cite[Proposition 3.3]{[Lin]}.
\begin{lem}\label{lem2.1}
Let $\mathcal{C}$ be an additive category, then

\textup{(1)} $\textup{Mor}(\mathcal{C})/\mathcal{R}\cong\textup{Mor}(\mathcal{C})/[\textup{s-epi}(\mathcal{C})]\cong \textup{mod-}\mathcal{C}$.

\textup{(2)} $\textup{Mor}(\mathcal{C})/\mathcal{R'}\cong\textup{Mor}(\mathcal{C})/[\textup{s-mono}(\mathcal{C})]\cong (\textup{mod-}\mathcal{C}^{\textup{op}})^{\textup{op}}$.
\end{lem}

\subsection{Extriangulated categories}
Let $\mathcal{C}$ be an additive category equipped with an additive bifunctor
%We denote by $Ab$ the category of abelian groups.
$\mathbb{E}: \mathcal{C}^{\textup{op}}\times\mathcal{C}\rightarrow Ab$. For any pair of objects $A,C\in\mathcal{C}$, an object $\delta\in\mathbb{E}(C,A)$ is called an $\mathbb{E}$-$extension$. For any morphism $a\in\mathcal{C}(A,A')$ and $c\in\mathcal{C}(C',C)$, we denote the $\mathbb{E}$-extension $\mathbb{E}(C,a)(\delta)\in\mathbb{E}(C,A')$ by $a_*\delta$ and denote the $\mathbb{E}$-extension $\mathbb{E}(c,A)(\delta)\in\mathbb{E}(C',A)$  by $c^*\delta$. Let $\delta\in \mathbb{E}(C,A)$ and $\delta'\in \mathbb{E}(C',A')$ be two $\mathbb{E}$-extensions. A $morphism$ $(a,c): \delta\rightarrow\delta'$ of $\mathbb{E}$-extensions is a pair of morphisms $a\in\mathcal{C}(A,A')$ and $c\in\mathcal{C}(C,C')$ such that $a_*\delta=c^*\delta'$.

Let $A,C\in \mathcal{C}$ be any pair of objects.  Two sequences of morphisms in $\mathcal{C}$
$$A\xrightarrow{x} B\xrightarrow{y} C\ \ \textup{and} \ \  A\xrightarrow{x'} B'\xrightarrow{y'} C$$
are $equivalent$ if there exists an isomorphism $b\in\mathcal{C}(B,B')$ such that the following diagram is commutative.
$$\xymatrix{
A\ar[r]^{x} \ar@{=}[d]& B\ar[r]^{y}\ar[d]^{b}_{\simeq} & C\ar@{=}[d]\\
A\ar[r]^{x'} & B'\ar[r]^{y'} & C\\
}$$
We denote the equivalence class of $A\xrightarrow{x} B\xrightarrow{y} C$ by $[A\xrightarrow{x} B\xrightarrow{y} C]$.

\begin{defn} Let $\mathbb{E}: \mathcal{C}^{\textup{op}}\times\mathcal{C}\rightarrow Ab$ be an additive bifunctor.
A correspondence $\mathfrak{s}$ is called a $realization$ of $\mathbb{E}$ if it associates an equivalence class $\mathfrak{s}(\delta)=[A\xrightarrow{x} B\xrightarrow{y} C]$ to any $\mathbb{E}$-extension $\delta\in \mathbb{E}(C,A)$ and associates a commutative diagram $$\xymatrix{
A\ar[r]^{x} \ar[d]^{a}& B\ar[r]^{y}\ar[d]^{b} & C\ar[d]^{c}\\
A\ar[r]^{x'} & B'\ar[r]^{y'} & C\\
}$$ to any morphism $(a,c):\delta\rightarrow\delta'$ of $\mathbb{E}$-extensions, where $\mathfrak{s}(\delta)=[A\xrightarrow{x} B\xrightarrow{y} C]$ and $\mathfrak{s}(\delta')=[A'\xrightarrow{x'} B'\xrightarrow{y'} C']$. In the above situation, we say the sequence $A\xrightarrow{x} B\xrightarrow{y} C$ $realizes$ $\delta$ and the triple $(a,b,c)$ $realizes$ $(a,c)$.
\end{defn}

\begin{defn}
A realization $\mathfrak{s}$ of $\mathbb{E}$ is said to be $additive$ if it satisfies the following two conditions.

(1) Assume that $0 \in \mathbb{E}(C,A)$ is the zero element, %for any $A,C\in\mathcal{C}$,
then $\mathfrak{s}(0)=[A\xrightarrow{\left(
                                                                   \begin{smallmatrix}
                                                                     1 \\
                                                                     0 \\
                                                                   \end{smallmatrix}
                                                                 \right)
} A\oplus C\xrightarrow{(0,1)} C]$.

(2) Assume that $\mathfrak{s}(\delta)=[A\xrightarrow{x} B\xrightarrow{y} C]$ and $\mathfrak{s}(\delta')=[A'\xrightarrow{x'} B'\xrightarrow{y'} C']$, then $\mathfrak{s}(\delta\oplus\delta')=[A\oplus A'\xrightarrow{x\oplus x'} B\oplus B'\xrightarrow{y\oplus y'} C\oplus C']$, where $\delta\oplus\delta'\in \mathbb{E}(C\oplus C',A\oplus A')$ is the element corresponding to $(\delta,0,0,\delta')$ under the isomorphism $\mathbb{E}(C\oplus C',A\oplus A')\cong\mathbb{E}(C,A)\oplus\mathbb{E}(C,A')\oplus\mathbb{E}(C',A)\oplus\mathbb{E}(C',A')$.
\end{defn}

Let $\mathfrak{s}$ be an additive realization of $\mathbb{E}$. If $\mathfrak{s}(\delta)=[A\xrightarrow{x} B\xrightarrow{y} C]$, then the sequence $A\xrightarrow{x} B\xrightarrow{y} C$ is called an $\mathfrak{s}$-$conflation$, the morphism $x$ is called an $\mathfrak{s}$-$inflation$ and $y$ is called an $\mathfrak{s}$-$deflation$. In this case, we say
$\xymatrixrowsep{0.05pc}\xymatrix{A\ar[r]^{x}&B\ar[r]^{y} & C\ar@{-->}[r]^{\delta}&}$ is an $\mathfrak{s}$-$triangle$. Let $\xymatrixrowsep{0.05pc}\xymatrix{A\ar[r]^{x}&B\ar[r]^{y} & C\ar@{-->}[r]^{\delta}&}$ and $\xymatrixrowsep{0.05pc}\xymatrix{A'\ar[r]^{x'}&B'\ar[r]^{y'} & C'\ar@{-->}[r]^{\delta'}&}$ be any pair of $\mathfrak{s}$-triangles. Let $(a,c):\delta\rightarrow\delta'$ be a morphism of $\mathbb{E}$-extensions. If a triple $(a,b,c)$ realizes $(a,c)$, then we say $(a,b,c)$ is a $morphism$ of $\mathfrak{s}$-triangles.

\begin{defn} (\cite[Definition 2.12]{[NP]})
A triple $(\mathcal{C}, \mathbb{E}, \mathfrak{s})$ is an {\em extriangulated category} if the following conditions are satisfied.

(ET1)  $\mathbb{E}: \mathcal{C}^{\textup{op}}\times\mathcal{C}\rightarrow Ab$ is an additive bifunctor.

(ET2) $\mathfrak{s}$ is an additive realization of $\mathbb{E}$.

(ET3) Each commutative diagram
$$\xymatrix{
 A\ar[r]^{x}\ar[d]^{a} & B \ar[r]^{y}\ar[d]^{b} &C\ar@{-->}[r]^{\delta} & \\
A'\ar[r]^{x'} & B' \ar[r]^{y'} &  C' \ar@{-->}[r]^{\delta'}&
}$$
whose rows are $\mathfrak{s}$-triangles can be completed to a morphism of  $\mathfrak{s}$-triangles.

(ET3)$^{\textup{op}}$ Each commutative diagram
$$\xymatrix{
 A\ar[r]^{x} & B \ar[r]^{y}\ar[d]^{b} &C\ar[d]^{c}\ar@{-->}[r]^{\delta} & \\
A'\ar[r]^{x'} & B' \ar[r]^{y'} &  C' \ar@{-->}[r]^{\delta'}&
}$$
whose rows are $\mathfrak{s}$-triangles can be completed to a morphism of  $\mathfrak{s}$-triangles.

(ET4) Let $\xymatrixrowsep{0.05pc}\xymatrix{A\ar[r]^{f}&B\ar[r]^{f'} & D\ar@{-->}[r]^{\delta}&}$  and $\xymatrixrowsep{0.05pc}\xymatrix{B\ar[r]^{g}&C\ar[r]^{g'} & F\ar@{-->}[r]^{\delta'}&}$  be $\mathfrak{s}$-triangles.
There exists a commutative diagram
$$\xymatrix{
A\ar[r]^{f}\ar@{=}[d] & B \ar[r]^{f'}\ar[d]^{g} & D\ar@{-->}[r]^{\delta}\ar[d]^{d} & \\
A\ar[r]^{h} & C \ar[d]^{g'}\ar[r]^{h'} &  E \ar[d]^{e}\ar@{-->}[r]^{\delta''} & \\
& F\ar@{=}[r]\ar@{-->}[d]^{\delta'} & F \ar@{-->}[d]^{f'_*\delta'}&\\
&&&\\
}$$
such that the second row and the third column are $\mathfrak{s}$-triangles, moreover, $\delta=d^*\delta''$ and $f_*\delta''=e^*\delta'$.

(ET4)$^{\textup{op}}$ Let $\xymatrixrowsep{0.05pc}\xymatrix{D\ar[r]^{f'}&A\ar[r]^{f} & B\ar@{-->}[r]^{\delta}&}$  and $\xymatrixrowsep{0.05pc}\xymatrix{F\ar[r]^{g'}&B\ar[r]^{g} & C\ar@{-->}[r]^{\delta'}&}$  be $\mathfrak{s}$-triangles.
There exists a commutative diagram
$$\xymatrix{
D\ar[r]^{d}\ar@{=}[d] & E \ar[r]^{e}\ar[d]^{h'} & F\ar@{-->}[r]^{g'^*\delta}\ar[d]^{g'} & \\
D\ar[r]^{f'} & A \ar[d]^{h}\ar[r]^{f} &  B \ar[d]^{g}\ar@{-->}[r]^{\delta} & \\
& C\ar@{=}[r]\ar@{-->}[d]^{\delta''} & C \ar@{-->}[d]^{\delta'}&\\
&&&\\
}$$
such that the first row and the second column are $\mathfrak{s}$-triangles, moreover, $\delta'=e_*\delta''$ and $d_*\delta=g^*\delta''$.
\end{defn}

\begin{defn}
Let $(\mathcal{C}, \mathbb{E}, \mathfrak{s})$ be an extriangulated category.

(1) An object $P\in\mathcal{C}$ is called {\em projective} if for any $\mathfrak{s}$-deflation $y:B\rightarrow C$
  and any morphism $c:P\rightarrow C$, there exists a morphism $b:P\rightarrow B$ such that $yb=c$. The full subcategory of projectives is denoted  by $\mathcal{P}$.

(2) We say that $\mathcal{C}$ {\em has enough projectives} if for any object $C\in\mathcal{C}$ there exists an $\mathfrak{s}$-triangle $\xymatrixrowsep{0.05pc}\xymatrix{A\ar[r]^{x}&P\ar[r]^{y} & C\ar@{-->}[r]^{\delta}&}$ with $P\in\mathcal{P}$.
\end{defn}

\begin{example}\cite[Example 3.26]{[NP]}
(1) Let $\mathcal{C}$ be an exact category, then $\mathcal{C}$ is an extriangulated category with $\mathbb{E}(-,-)=\textup{Ext}^1_\mathcal{C}(-,-)$. In particular, if $\mathcal{C}$ is an exact category with enough projectives, then $\mathcal{C}$ is an extriangulated category with enough projectives.

(2) Let $\mathcal{C}$ be a triangulated category with shift functor $[1]$, then $\mathcal{C}$ is an extriangulated category with $\mathbb{E}(-,-)=\mathcal{C}(-,-[1])$. Moreover, $\mathcal{C}$ has enough projectives. In this case, $\mathcal{P}$ consists of zero objects.
\end{example}

The following lemmas will be used frequently later.

\begin{lem}(\cite[Corollary 3.5]{[NP]})\label{2.1}
Let $\mathcal{C}$ be an extriangulated category. Assume that the diagram $$\xymatrix{
 A\ar[r]^{x}\ar[d]^{a} & B \ar[r]^{y}\ar[d]^{b} &C\ar[d]^{c}\ar@{-->}[r]^{\delta} & \\
A'\ar[r]^{x'} & B' \ar[r]^{y'} &  C' \ar@{-->}[r]^{\delta'}&
}$$ is  a morphism of $\mathfrak{s}$-triangles. Then the following statements are equivalent.

\textup{(1)} $a$ factors through $x$.

\textup{(2)} $a_*\delta=c^*\delta'=0$.

\textup{(3)} $c$ factors through $y'$.
\end{lem}

\begin{lem} (\cite[Proposition 1.20]{[LN]})\label{2.2}
Let $\mathcal{C}$ be an extriangulated category.  Assume that $\xymatrixrowsep{0.05pc}\xymatrix{A\ar[r]^{x}&B\ar[r]^{y} & C\ar@{-->}[r]^{\delta}&}$ is an $\mathfrak{s}$-triangle, $f: A\rightarrow D$ is a morphism and $\xymatrixrowsep{0.05pc}\xymatrix{D\ar[r]^{d}&E\ar[r]^{e} & C\ar@{-->}[r]^{f_*\delta}&}$  is an $\mathfrak{s}$-triangle, then there is a morphism $g:B\rightarrow E$ which gives a morphism of  $\mathfrak{s}$-triangles
$$\xymatrix{
 A\ar[r]^{x}\ar[d]^{f} & B \ar[r]^{y}\ar@{-->}[d]^{g} &C\ar@{-->}[r]^{\delta}\ar@{=}[d] & \\
D\ar[r]^{d} & E \ar[r]^{e} &  C \ar@{-->}[r]^{f_*\delta}&
}$$ and moreover, $\xymatrix{A\ar[r]^{\left(
                                                           \begin{smallmatrix}
                                                             f \\
                                                             x \\
                                                           \end{smallmatrix}
                                                         \right)
} & D\oplus B\ar[r]^{(d, -g)} & E\ar@{-->}[r]^{e^*\delta}& }$ is an $\mathfrak{s}$-triangle.
\end{lem}

\begin{lem}(\cite[Corollary 3.12]{[NP]})\label{lem2.0}
Let $\mathcal{C}$ be an extriangulated category. Then for any $\mathfrak{s}$-triangle $\xymatrixrowsep{0.05pc}\xymatrix{A\ar[r]^{x}&B\ar[r]^{y} & C\ar@{-->}[r]^{\delta}&}$, the following two sequences are exact.
$$\mathcal{C}(-,A)\rightarrow\mathcal{C}(-,B)\rightarrow\mathcal{C}(-,C)\rightarrow\mathbb{E}(-,A)\rightarrow\mathbb{E}(-,B)\rightarrow\mathbb{E}(-,C),$$
$$\mathcal{C}(C,-)\rightarrow\mathcal{C}(B,-)\rightarrow\mathcal{C}(A,-)\rightarrow\mathbb{E}(C,-)\rightarrow\mathbb{E}(B,-)\rightarrow\mathbb{E}(A,-).$$
\end{lem}

Let $\mathcal{C}$ be an extriangulated category with enough projectives $\mathcal{P}$ and enough injectives $\mathcal{I}$. Let $X$ be any object in $\mathcal{C}$. It admits an $\mathfrak{s}$-triangle
$$\xymatrixrowsep{0.05pc}\xymatrix{X\ar[r]&I^{0}\ar[r] & \Sigma X\ar@{-->}[r]^{\delta^X}&} (\textup{resp.} \xymatrixrowsep{0.05pc}\xymatrix{\Omega X\ar[r]&P_{0}\ar[r] &  X\ar@{-->}[r]^{\delta_X}&})$$ with $I^0\in\mathcal{I}$ (resp. $P_0\in\mathcal{P}$). We can get $\mathfrak{s}$-triangles $$\xymatrixrowsep{0.05pc}\xymatrix{\Sigma^iX\ar[r]&I^{i}\ar[r] & \Sigma^{i+1} X\ar@{-->}[r]^{\delta^{\Sigma^iX}}&} (\textup{resp.} \xymatrixrowsep{0.05pc}\xymatrix{\Omega^{i+1}X\ar[r]&P_{i}\ar[r] &  \Omega^iX\ar@{-->}[r]^{\delta_{\Omega^{i}X}}&})$$ with %$\Sigma^{i+1}X=\Sigma(\Sigma^iX)$ (resp. $\Omega^{i+1}X=\Omega(\Omega^iX)$) and
$I^i\in\mathcal{I}$ (resp. $P_i\in\mathcal{P}$) for $i>0$ recursively.

For convenience, we denote $\mathbb{E}(X,\Sigma^iY)\cong \mathbb{E}(\Omega^iX, Y)$ by $\mathbb{E}^i(X,Y)$, where the equivalence follows from \cite[Lemma 5.1]{[LN]}.

The following result extends the exact sequences appeared in Lemma \ref{lem2.0}.

\begin{lem}(\cite[Proposition 5.2]{[LN]})\label{lem2.5}
Let $\mathcal{C}$ be an extriangulated category with enough projectives $\mathcal{P}$ and enough injectives $\mathcal{I}$. Then for any $\mathfrak{s}$-triangle $$\xymatrixrowsep{0.05pc}\xymatrix{A\ar[r]^{x}&B\ar[r]^{y} & C\ar@{-->}[r]^{\delta}&},$$
the following two sequences are exact.
$$\mathcal{C}(-,A)\rightarrow\mathcal{C}(-,B)\rightarrow\mathcal{C}(-,C)\rightarrow\mathbb{E}(-,A)\rightarrow\mathbb{E}(-,B)\rightarrow\mathbb{E}(-,C)
\rightarrow\mathbb{E}^2(-, A)$$
$$\rightarrow\mathbb{E}^2(-, B)\rightarrow\mathbb{E}^2(-,C)\rightarrow\cdots\rightarrow \mathbb{E}^i(-,A)\rightarrow\mathbb{E}^i(-,B)\rightarrow\mathbb{E}^i(-,C)\rightarrow\cdots,$$
$$\mathcal{C}(C,-)\rightarrow\mathcal{C}(B,-)\rightarrow\mathcal{C}(A,-)\rightarrow\mathbb{E}(C,-)\rightarrow\mathbb{E}(B,-)
\rightarrow\mathbb{E}(A,-)\rightarrow\mathbb{E}^2(C,-)$$
$$\rightarrow\mathbb{E}^2(B,-)\rightarrow\mathbb{E}^2(A,-)\rightarrow\cdots\rightarrow
\mathbb{E}^i(C,-)\rightarrow\mathbb{E}^i(B,-)\rightarrow\mathbb{E}^i(A,-)\rightarrow\cdots.$$
\end{lem}

\begin{lem}\label{lem2.4} Let $\mathcal{C}$ be an extriangulated category with enough projectives $\mathcal{P}$ and $f:X\rightarrow Y$ be a morphism in $\mathcal{C}$.

\textup{(1)} If $\pi:P\rightarrow Y$ is an $\mathfrak{s}$-deflation with $P\in\mathcal{P}$, then $(f,-\pi): X\oplus P\rightarrow Y$ is an $\mathfrak{s}$-deflation and $\underline{(f,-\pi)}\cong\underline{f}$ in \textup{Mor-}$\mathcal{C}/[\mathcal{P}]$.

\textup{(2)} If $h:X\rightarrow Z$ is a morphism in $\mathcal{C}$ and $g:Y\rightarrow Z$ is an $\mathfrak{s}$-deflation such that $\underline{gf}=\underline{h}$ in \textup{Mor-}$\mathcal{C}/[\mathcal{P}]$, then there exists an object $P\in\mathcal{P}$ and two morphisms $u:X\rightarrow P$ and $v:P\rightarrow Y$ such that $g(f-vu)=h$.
\end{lem}

\begin{proof}
(1) The first assertion follows from \cite[Corollary 3.16]{[NP]} or the dual of Lemma \ref{2.2}. The second assertion is clear.

(2) Since $\underline{gf}=\underline{h}$,  there is an object $P\in\mathcal{P}$ and two morphisms $u:X\rightarrow P$ and $w:P\rightarrow Z$ such that $gf-h=wu$. Since $g:Y\rightarrow Z$ is an $\mathfrak{s}$-deflation, there exists a morphism $v:P\rightarrow Y$ such that $w=gv$. Therefore, $g(f-vu)=h$.
\end{proof}

\section{Proof of Theorem 1.1}

Throughout this paper, we assume that $(\mathcal{C}, \mathbb{E}, \mathfrak{s})$ is an extriangulated category.

%\subsection{Abelian quotients arising from category of $\mathfrak{s}$-deflations}

Let $\mathcal{M}$ be a full subcategory of $\mathcal{C}$. We denote by $\mathfrak{s}$-def$(\mathcal{M})$ (resp. $\mathfrak{s}\textup{-inf}(\mathcal{M})$) the full subcategory of $\textup{Mor}(\mathcal{M})$ consisting of $\mathfrak{s}$-deflations (resp. $\mathfrak{s}$-inflations). Recall that the full subcategory of $\mathfrak{s}$-def$(\mathcal{M})$ consisting of split epimorphisms (resp. split monomorphisms) is denoted by s-epi$(\mathcal{M})$ (resp. s-mono$(\mathcal{M})$).  We denote by sp-epi$(\mathcal{M})$ (resp. si-mono$(\mathcal{M})$) the full subcategory of $\mathfrak{s}$-def$(\mathcal{M})$ consisting of $(M\xrightarrow{1}M)\oplus(P\rightarrow M')$ (resp. $(M\xrightarrow{1}M)\oplus(M'\rightarrow I)$) with $P\in\mathcal{P}$ (resp. $I\in\mathcal{I}$).

\begin{lem}\label{lem3.0}
Let $\mathcal{C}$ be an extriangulated category with enough projectives $\mathcal{P}$ and $\mathcal{M}$ be a full subcategory of $\mathcal{C}$ containing $\mathcal{P}$. Assume that the following
$$\xymatrix{
 X\ar[r]^{k}\ar[d]^{g} & M_1 \ar[r]^{f}\ar[d]^{a} & M_2\ar@{-->}[r]^{\delta}\ar[d]^{b} & \\
 X'\ar[r]^{k'} & M_1' \ar[r]^{f'} &  M_2' \ar@{-->}[r]^{\delta'}&
}$$
is a morphism of $\mathfrak{s}$-triangles with $M_i, M_i'\in\mathcal{M}$. Then

\textup{(1)} The following statements are equivalent.

\ \ \ \textup{(a)} The morphism $\underline{b}$ factors through $\underline{f}'$ in $\mathcal{M}/[\mathcal{P}]$.

\ \ \ \textup{(b)} The morphism $b$ factors through $f'$.

\ \ \ \textup{(c)} The morphism $(a,b)$ factors through some object in $\textup{s-epi}(\mathcal{M})$.

\textup{(2)} The following statements are equivalent.

\ \ \ \textup{(a)} The morphism $\underline{a}$ factors through $\underline{f}$ in $\mathcal{M}/[\mathcal{P}]$.

 \ \ \ \textup{(b)} The morphism $(a,b)$ factors through some object in $\textup{sp-epi}(\mathcal{M})$.
\end{lem}

\begin{proof}
(1) Since (b)$\Leftrightarrow$(c) follows from Lemma \ref{lem2.1}  and (b)$\Rightarrow$(a) is clear, we only prove (a)$\Rightarrow$(b).
Suppose that there is a morphism $\underline{p}:M_2\rightarrow M_1'$ such that $\underline{f}'\underline{p}=\underline{b}$. By Lemma \ref{lem2.4}, there exists an object $P\in\mathcal{P}$ and two morphisms $u:M_2\rightarrow P$ and $v: P\rightarrow M_1'$ such that $f'(p-vu)=b$. Thus $b$ factors through $f'$.

(2) (a)$\Rightarrow$(b). Suppose that there is a morphism $\underline{p}:M_2\rightarrow M_1'$ such that $\underline{p}\underline{f}=\underline{a}$.
Since $\mathcal{C}$ has enough projectives, there is an $\mathfrak{s}$-deflation $a_1:P\rightarrow M_1'$ with $P\in\mathcal{P}$. It is easy to see that $a-pf$ factors through $a_1$. We assume that $a-pf=a_1a_2$ where $a_2:M_1\rightarrow P$. Since $(b-f'p)f=f'a-f'pf=f'a_1a_2$, we have the following commutative diagram.
$$\xymatrix{ M_1\ar[rr]^{f}\ar[dd]^{a}\ar[rd]^{\left(
                      \begin{smallmatrix}
                        f \\
                        a_2 \\
                      \end{smallmatrix}
                    \right)
}& & M_2 \ar[dd]^{b}\ar[rd]^{\left(
                      \begin{smallmatrix}
                        1 \\
                        b-f'p \\
                      \end{smallmatrix}
                    \right)
} &\\
& M_2\oplus P \ar[rr]^{\ \ \ \ \ \ \ \ \ \left(
                      \begin{smallmatrix}
                        1 & 0 \\
                        0 & f'a_1 \\
                      \end{smallmatrix}
                    \right)
}\ar[ld]^{(p,a_1)} & &M_2\oplus M_2'\ar[ld]^{(f'p,1)}\\
 M_1'\ar[rr]^{f'}& & M_2'&
}$$ In other words, $(a,b)$ factors through $(M_2\oplus P\xrightarrow{\left(
                      \begin{smallmatrix}
                        1 & 0 \\
                        0 & f'a_1 \\
                      \end{smallmatrix}
                    \right)} M_2\oplus M_2')\in \textup{sp-epi}(\mathcal{M})$.

(b)$\Rightarrow$(a). Assume that the morphism $(a,b)$ factors through $(M\oplus P\xrightarrow{\left(
                      \begin{smallmatrix}
                        1 & 0 \\
                        0 & \pi \\
                      \end{smallmatrix}
                    \right)} M\oplus M')\in\textup{sp-epi}(\mathcal{M})$. Suppose that the following diagram
$$\xymatrix{ M_1\ar[rr]^{f}\ar[dd]^{a}\ar[rd]^{\left(
                      \begin{smallmatrix}
                        a_1 \\
                        a'_1 \\
                      \end{smallmatrix}
                    \right)
}& & M_2 \ar[dd]^{b}\ar[rd]^{\left(
                      \begin{smallmatrix}
                        b_1 \\
                        b_1' \\
                      \end{smallmatrix}
                    \right)
} &\\
& M\oplus P \ar[rr]^{\ \ \ \ \ \ \ \ \ \left(
                      \begin{smallmatrix}
                        1 & 0 \\
                        0 & \pi \\
                      \end{smallmatrix}
                    \right)
}\ar[ld]^{(a_2,a_2')} & &M\oplus M'\ar[ld]^{(b_2,b_2')}\\
 M_1'\ar[rr]^{f'}& & M_2'&
}$$ is commutative. Let $p=a_2b_1: M_2\rightarrow M_1'$, then $pf=a_2b_1f=a_2a_1$, thus $\underline{a}=\underline{a_2}\underline{a_1}=\underline{p}\underline{f}$.
\end{proof}

\begin{thm}\label{3.1}
Let $\mathcal{C}$ be an extriangulated category and $\mathcal{M}$ be a full subcategory of $\mathcal{C}$.

\textup{(1)} If $\mathcal{C}$ has enough projectives $\mathcal{P}$ and $\mathcal{M}$ contains $\mathcal{P}$, then   $\mathfrak{s}\textup{-def}(\mathcal{M})/[\textup{s-epi}(\mathcal{M})]\cong \textup{mod-}\mathcal{M}/[\mathcal{P}]$ and  $\mathfrak{s}\textup{-def}(\mathcal{M})/[\textup{sp-epi}(\mathcal{M})]\cong (\textup{mod-}(\mathcal{M}/[\mathcal{P}])^{\textup{op}})^{\textup{op}}$.

\textup{(2)} If $\mathcal{C}$ has enough injectives $\mathcal{I}$ and $\mathcal{M}$ contains $\mathcal{I}$, then
$\mathfrak{s}\textup{-inf}(\mathcal{M})/[\textup{s-mono}(\mathcal{M})]\cong (\textup{mod-}(\mathcal{M}/[\mathcal{I}])^{\textup{op}})^{\textup{op}}$ and  $\mathfrak{s}\textup{-inf}(\mathcal{M})/[\textup{si-mono}(\mathcal{M})]\cong \textup{mod-}\mathcal{M}/[\mathcal{I}]$.
%In particular, we have $\mathfrak{s}\textup{-def}(\mathcal{C})/[\textup{s-epi}(\mathcal{C})]\cong \textup{mod-}\mathcal{C}/[\mathcal{P}].$
\end{thm}

\begin{proof} Since (2) is dual to (1), we only prove (1).

Define a functor $$F: \mathfrak{s}\textup{-def}(\mathcal{M})\rightarrow\textup{Mor}(\mathcal{M}/[\mathcal{P}]),\ \ (M_1\xrightarrow{f}M_2) \mapsto (M_1\xrightarrow{\underline{f}}M_2).$$ For any object $\underline{f}:M_1\rightarrow M_2$ in $\textup{Mor}(\mathcal{M}/[\mathcal{P}])$, by Lemma \ref{lem2.4} there is an object $P\in\mathcal{P}$ and an $\mathfrak{s}$-deflation $(f,-\pi):M_1\oplus P\rightarrow M_2$  such that $\underline{(f,-\pi)}\cong\underline{f}$. Therefore, $F(f,-\pi)\cong \underline{f}$ and $F$ is dense.

Assume that $f:M_1\rightarrow M_2$ and $f':M'_1\rightarrow M'_2$ are objects in $\mathfrak{s}\textup{-def}(\mathcal{M})$  and $(\underline{a},\underline{b})$ is a morphism in $\textup{Mor}(\mathcal{M}/[\mathcal{P}])$ from $\underline{f}$ to $\underline{f'}$.
Then $\underline{b}\underline{f}=\underline{f'}\underline{a}$. By Lemma \ref{lem2.4}, there exists an object $Q\in\mathcal{P}$ and two morphisms $u:M_1\rightarrow Q$ and $v:Q\rightarrow M_1'$ such that $f'(a-vu)=bf$.  Thus, $F(a-vu,b)=(\underline{a},\underline{b})$ and the functor $F$ is full.

 The  functor $F$ induces a full and dense functor $\widetilde{F}: \mathfrak{s}\textup{-def}(\mathcal{M})\rightarrow \textup{Mor}(\mathcal{M}/[\mathcal{P}])/\mathcal{R}$.
By Lemma \ref{lem3.0}(1), we have $\mathfrak{s}\textup{-def}(\mathcal{M})/[\textup{s-epi}(\mathcal{C})]\cong \textup{Mor}(\mathcal{M}/[\mathcal{P}])/\mathcal{R}$. It follows that $\mathfrak{s}\textup{-def}(\mathcal{M})/[\textup{s-epi}(\mathcal{M})]\cong \textup{mod-}\mathcal{M}/[\mathcal{P}]$ by Lemma \ref{lem2.1}(1).

 The  functor $F$ induces a full and dense functor $\widehat{F}: \mathfrak{s}\textup{-def}(\mathcal{M})\rightarrow \textup{Mor}(\mathcal{M}/[\mathcal{P}])/\mathcal{R}'$.
 By  Lemma \ref{lem3.0}(2), we have $\mathfrak{s}\textup{-def}(\mathcal{M})/[\textup{sp-epi}(\mathcal{C})]\cong \textup{Mor}(\mathcal{M}/[\mathcal{P}])/\mathcal{R}'$. It follows that $\mathfrak{s}\textup{-def}(\mathcal{M})/[\textup{sp-epi}(\mathcal{M})]\cong (\textup{mod-}(\mathcal{M}/[\mathcal{P}])^{\textup{op}})^{\textup{op}}$ by Lemma \ref{lem2.1}(2).
\end{proof}

%\begin{rem}
%Let $\mathcal{M}$ be a full and contravariantly finite subcategory of $\mathcal{C}$ containing $\mathcal{P}$. Then $\textup{mod-}\mathcal{M}/[\mathcal{P}]$ is abelian, thus $\mathfrak{s}\textup{-def}(\mathcal{M})/[\textup{s-epi}(\mathcal{M})]$ is abelian.
%\end{rem}

\section{Application to category of $\mathfrak{s}$-triangles}

In this section, we will investigate the first application of Theorem \ref{3.1} in the case when $\mathcal{M}=\mathcal{C}$.

 We denote by $\mathfrak{s}\textup{-tri}(\mathcal{C})$ the category of $\mathfrak{s}$-triangles in $\mathcal{C}$, where the objects are the $\mathfrak{s}$-triangles $X_\bullet=(\xymatrix{X_1\ar[r]^{f_1}&X_2\ar[r]^{f_2} & X_3\ar@{-->}[r]^{\delta}&})$ and the morphisms from $X_\bullet$ to $Y_\bullet$ are the triples $\varphi_\bullet=(\varphi_1,\varphi_2,\varphi_3)$ such that the following diagram is commutative
$$\xymatrix{
X_1 \ar[r]^{f_1} \ar[d]^{\varphi_1} & X_2 \ar[r]^{f_2}\ar[d]^{\varphi_2} & X_3 \ar@{-->}[r]^{\delta}\ar[d]^{\varphi_3} & \\
Y_1\ar[r]^{g_1} & Y_2\ar[r]^{g_2} & Y_3\ar@{-->}[r]^{\delta'} &
}$$
and $\varphi_{1\ast}\delta=\varphi_3^*\delta'$.
 Let $X_\bullet$ and $Y_\bullet$ be two $\mathfrak{s}$-triangles, we denote by $\mathcal{R}_2(X_\bullet,Y_\bullet)$ (resp. $\mathcal{R}'_1(X_\bullet,Y_\bullet)$)
  the class of morphisms $\varphi_\bullet: X_\bullet\rightarrow Y_\bullet$ such that $\varphi_3$ factors through $g_2$ %there is a morphism $p:X_3\rightarrow Y_2$ such that $g_2p=\varphi_3$.
 (resp. $\varphi_1$ factors through $f_1$).  It is easy to see that $\mathcal{R}_2$ and $\mathcal{R}'_1$ are ideals of $\mathfrak{s}\textup{-tri}(\mathcal{C})$.

\begin{thm}\label{thm3.2}
Let $\mathcal{C}$ be an extriangulated category.

\textup{(1)} If $\mathcal{C}$ has enough projectives $\mathcal{P}$, then $\mathfrak{s}\textup{-tri}(\mathcal{C})/\mathcal{R}_2\cong \textup{mod-}\mathcal{C}/[\mathcal{P}]$.

\textup{(2)} If $\mathcal{C}$ has enough injectives $\mathcal{I}$, then $\mathfrak{s}\textup{-tri}(\mathcal{C})/\mathcal{R}_2\cong (\textup{mod-}(\mathcal{C}/[\mathcal{I}])^{\textup{op}})^{\textup{op}}$.
\end{thm}

\begin{proof}
(1) We have $\mathfrak{s}\textup{-tri}(\mathcal{C})/\mathcal{R}_2\cong \mathfrak{s}\textup{-def}(\mathcal{C})/[\textup{s-epi}(\mathcal{C})]$ by Lemma \ref{lem3.0}. Thus $\mathfrak{s}\textup{-tri}(\mathcal{C})/\mathcal{R}_2\cong \textup{mod-}\mathcal{C}/[\mathcal{P}]$ follows from Theorem \ref{3.1}(1).

(2) We note that $\mathcal{R}_2=\mathcal{R}_1'$ by Lemma \ref{2.1}. Thus $\mathfrak{s}\textup{-tri}(\mathcal{C})/\mathcal{R}_2=\mathfrak{s}\textup{-tri}(\mathcal{C})/\mathcal{R}'_1
\cong\mathfrak{s}\textup{-inf}(\mathcal{C})/[\textup{s-mono}(\mathcal{C})]\cong (\textup{mod-}(\mathcal{C}/[\mathcal{I}])^{\textup{op}})^{\textup{op}}$,
where the last equivalence follows from Theorem \ref{3.1}(2).
\end{proof}

\begin{lem}  Let $\mathcal{C}$ be an extriangulated category.
 Assume that the following
$$\xymatrix{X_\bullet\ar[d]^{\varphi_\bullet}  & X_1 \ar[r]^{f_1}\ar[d]^{\varphi_1} & X_2 \ar[r]^{f_2}\ar[d]^{\varphi_2} & X_3 \ar@{-->}[r]^{\delta}\ar[d]^{\varphi_3} & \\
Y_\bullet  & Y_1 \ar[r]^{g_1} & Y_2 \ar[r]^{g_2} & Y_3 \ar@{-->}[r]^{\delta'} & \\
}$$
is a morphism of $\mathfrak{s}$-triangles. Then

\textup{(1)} The following statements are equivalent.

\ \ \ \textup{(a)} $\underline{\varphi_\bullet}=0$ in $\mathfrak{s}\textup{-tri}(\mathcal{C})/\mathcal{R}_2$.

\ \ \ \textup{(b)} $\varphi_1$ factors through $f_1$.

\ \ \ \textup{(c)} $\varphi_{3}$ factors through $g_2$.

\textup{(2)} The following statements are equivalent.

\ \ \ \textup{(a)} $\underline{\varphi_\bullet}$ is a monomorphism in $\mathfrak{s}\textup{-tri}(\mathcal{C})/\mathcal{R}_2$.

\ \ \ \textup{(b)} $\left(
                                                                         \begin{smallmatrix}
                                                                           f_1 \\
                                                                           \varphi_1 \\
                                                                         \end{smallmatrix}
                                                                       \right)
:X_1\rightarrow X_2\oplus Y_1$ is a section.
\end{lem}

\begin{proof}
(1) It follows from Lemma \ref{2.1}.

(2) The proof is similar to \cite[Lemma 4.7]{[Lin]}.
\end{proof}

If $\mathcal{C}$ has enough projectives $\mathcal{P}$, then $\mathfrak{s}\textup{-tri}(\mathcal{C})/\mathcal{R}_2\cong \textup{mod-}\mathcal{C}/[\mathcal{P}]$ is abelian by Theorem \ref{thm3.2}. The following result implies that $\mathfrak{s}\textup{-tri}(\mathcal{C})/\mathcal{R}_2$ is always abelian for general case.

\begin{prop}\label{prop1}
Let $\mathcal{C}
$ be an extriangulated category. Then $\mathfrak{s}\textup{-tri}(\mathcal{C})/\mathcal{R}_2$ is an abelian category. %In particular, if $\mathcal{C}$ has enough projective objects, then $\mathfrak{s}\textup{-tri}(\mathcal{C})/\mathcal{R}_2\cong \textup{mod-}\mathcal{C}/[\mathcal{P}]$.
\end{prop}

\begin{proof}
%Given a morphism $\varphi_\bullet: X_\bullet\rightarrow Y_\bullet$ in $\mathbb{E}(\mathcal{C})/\mathcal{R}$,
%We only need to prove the first assertion.
The proof is an adaption of \cite[Theorem 4.8]{[Lin]}. Assume that the following
$$\xymatrix{X_\bullet\ar[d]^{\varphi_\bullet}  & X_1 \ar[r]^{f_1}\ar[d]^{\varphi_1} & X_2 \ar[r]^{f_2}\ar[d]^{\varphi_2} & X_3 \ar@{-->}[r]^{\delta}\ar[d]^{\varphi_3} & \\
Y_\bullet  & Y_1 \ar[r]^{g_1} & Y_2 \ar[r]^{g_2} & Y_3 \ar@{-->}[r]^{\delta'} & \\
}$$
is a morphism of $\mathfrak{s}$-triangles. Thus $\varphi_{1\ast}\delta=\varphi_3^*\delta'$ by definition.  By Lemma \ref{2.2} and its dual, we have the following morphisms of $\mathfrak{s}$-triangles.
$$\xymatrix{
K(\varphi_\bullet)\ar[d]^{k_\bullet} & X_1 \ar[r]^{\left(
                                                                                        \begin{smallmatrix}
                                                                                          f_1  \\
                                                                                           \varphi_1\\
                                                                                        \end{smallmatrix}
                                                                                      \right)}\ar@{=}[d] & X_2\oplus Y_1 \ar[r]^{\left(
                                                                                        \begin{smallmatrix}
                                                                                          a_1 & -h_1 \\
                                                                                        \end{smallmatrix}
                                                                                      \right)}\ar[d]^{(1,0)
 } & Z \ar@{-->}[r]^{h_2^*\delta}\ar[d]^{h_2} & \\
X_\bullet\ar[d]^{\pi_\bullet} & X_1 \ar[r]^{f_1}\ar[d]^{\varphi_1} & X_2 \ar[r]^{f_2}\ar[d]^{a_1} & X_3 \ar@{-->}[r]^{\delta}\ar@{=}[d] &  \\
I(\varphi_\bullet)\ar[d]^{i_\bullet} &  Y_1 \ar[r]^{h_1}\ar@{=}[d] & Z \ar[r]^{h_2}\ar[d]^{a_2} & X_3 \ar@{-->}[r]^{\varphi_{1*}\delta}\ar[d]^{\varphi_3} & \\
Y_\bullet\ar[d]^{c_\bullet} &  Y_1 \ar[r]^{g_1}\ar[d]^{h_1} & Y_2 \ar[r]^{g_2}\ar[d]^{\left(
                                                                                                                            \begin{smallmatrix}
                                                                                                                              0 \\
                                                                                                                              1 \\
                                                                                                                            \end{smallmatrix}
                                                                                                                          \right)} & Y_3 \ar@{-->}[r]^{\delta'}\ar@{=}[d] & \\
C(\varphi_\bullet) &   Z \ar[r]^{\left(
                                                                                        \begin{smallmatrix}
                                                                                          h_2  \\
                                                                                          a_2\\
                                                                                        \end{smallmatrix}
                                                                                      \right)} & X_3\oplus Y_2 \ar[r]^{(-\varphi_3, g_2)} & Y_3 \ar@{-->}[r]^{h_{1*}\delta'} & \\
}$$
Moreover, we have $\underline{\varphi_\bullet}=\underline{i_\bullet\pi_\bullet}$ in $\mathfrak{s}\textup{-tri}(\mathcal{C})/\mathcal{R}_2$.
It is routine to check that $k_{\bullet}:K(\varphi_\bullet)\rightarrow X_\bullet$ is a kernel of $\underline{\varphi_\bullet}$, $c_\bullet: Y_\bullet \rightarrow C(\varphi_\bullet)$ is a cokernel of $\underline{\varphi_\bullet}$ and $$\textup{Coker}(\textup{Ker}(\underline{\varphi_\bullet}))\cong I(\varphi_\bullet)\cong \textup{Ker}(\textup{Coker}(\underline{\varphi_\bullet})).$$
\end{proof}

\begin{prop}\label{prop3.2}
Let  $\mathcal{C}$ be an extriangulated category. Then an $\mathfrak{s}$-triangle $P_X=(\xymatrix{\Omega X\ar[r]^{f_1}&P\ar[r]^{f_2} & X\ar@{-->}[r]^{\rho}&}$) with $P\in\mathcal{P}$ is a projective object in $\mathfrak{s}\textup{-tri}(\mathcal{C})/\mathcal{R}_2$. Moreover, if $\mathcal{C}$ has enough projectives, then each projective object in $\mathfrak{s}\textup{-tri}(\mathcal{C})/\mathcal{R}_2$ is of the form $P_X$.
\end{prop}

\begin{proof}
The proof is an adaption of \cite[Proposition 4.11]{[Lin]}. We omit it.
\end{proof}

\begin{defn} (\cite{[INP],[ZZ]})
An $\mathfrak{s}$-triangle $\xymatrix{X_1\ar[r]^{f_1}& X_2\ar[r]^{f_2} & X_3\ar@{-->}[r]^{\delta}&}$ is called {\em Auslander-Reiten $\mathfrak{s}$-triangle} if the following holds:

(1) $\delta\in\mathbb{E}(C,A)$ is non-split.

(2)  If $g: X_1\rightarrow Y$  is not a section, then $g$ factors through $f_1$.

(3)  If $h:Z\rightarrow X_3$   is not a retraction, then $h$ factors through $f_2$.
\end{defn}

\begin{prop}\label{prop3.2.0}
Let $\mathcal{C}$ be a Krull-Smidt extriangulated category. Assume that $X_\bullet: \xymatrix{X_{1}\ar[r]^{f_1}& X_2\ar[r]^{f_2} & X_3\ar@{-->}[r]^{\delta}&}$ is a non-split $\mathfrak{s}$-triangle such that $X_1$ and $X_2$ are indecomposable. Then $X_\bullet$ is a simple object in $\mathfrak{s}\textup{-tri}(\mathcal{C})/\mathcal{R}_2$ if and only if $X_\bullet$ is an Auslander-Reiten $\mathfrak{s}$-triangle in $\mathcal{C}$.
\end{prop}

\begin{proof}
The proof is an adaption of \cite[Theorem 4.20 (a)]{[Lin]}.
\end{proof}

From now on to the end of this section we assume that $\mathcal{C}$ is an extriangulated category with enough projectives $\mathcal{P}$ and enough injectives $\mathcal{I}$.

Given an $\mathfrak{s}$-triangle  $\delta=(\xymatrix{X_1\ar[r]^{f_1}&X_2\ar[r]^{f_2} & X_3\ar@{-->}[r]^{\rho}&})$, we define  the {\em contravariant defect} $\delta^\ast$  and the {\em covariant defect} $\delta_\ast$ by the following exact sequence of functors
$$\mathcal{C}(-,X_{1})\xrightarrow{\mathcal{C}(-,f_{1})}\mathcal{C}(-,X_{2})\xrightarrow{\mathcal{C}(-,f_{2})}\mathcal{C}(-,X_3)\rightarrow \delta^\ast\rightarrow 0,$$
$$\mathcal{C}(X_{3},-)\xrightarrow{\mathcal{C}(f_{2},-)}\mathcal{C}(X_{2},-)\xrightarrow{\mathcal{C}(f_{1},-)}\mathcal{C}(X_1,-)\rightarrow \delta_\ast\rightarrow 0.$$

\begin{example}\label{ex3.2}
(1) Let $\delta=P_X=(\xymatrix{\Omega X\ar[r]^{f_1}&P\ar[r]^{f_2} & X\ar@{-->}[r]^{\rho}&})$ with $P\in\mathcal{P}$. Then  $\delta^\ast=\mathcal{C}/[\mathcal{P}](-,X)$ and $\delta_\ast=\mathbb{E}(X,-)$.

(2) Let $\delta=I_X=(\xymatrix{X\ar[r]^{f_1}& I\ar[r]^{f_2} & \Sigma X\ar@{-->}[r]^{\rho}&})$ with $I\in\mathcal{I}$. Then $\delta^\ast=\mathbb{E}(-,X)$ and $\delta_\ast=\mathcal{C}/[\mathcal{I}](X,-)$.
\end{example}

The following result gives an explanation of \cite[Theorem 4.1]{[INP]}.

\begin{thm}\label{rem3.2} Let $\mathcal{C}$ be an extriangulated category with enough projectives $\mathcal{P}$ and enough injectives $\mathcal{I}$.

\textup{(1)} We have the following equivalences
 $$\mathfrak{s}\textup{-tri}(\mathcal{C})/\mathcal{R}_2\cong \textup{mod-}\mathcal{C}/[\mathcal{P}]\cong(\textup{mod-}(\mathcal{C}/[\mathcal{I}])^{\textup{op}})^{\textup{op}}.$$
 Moreover, the equivalence $F:\mathfrak{s}\textup{-tri}(\mathcal{C})/\mathcal{R}_2\cong \textup{mod-}\mathcal{C}/[\mathcal{P}]$ is given by $\delta\mapsto \delta^*$ and the equivalence $G:\mathfrak{s}\textup{-tri}(\mathcal{C})/\mathcal{R}_2\cong (\textup{mod-}(\mathcal{C}/[\mathcal{I}])^{\textup{op}})^{\textup{op}}$ is given by $\delta\mapsto \delta_*$.

\textup{(2)} The abelian category $\textup{mod-}\mathcal{C}/[\mathcal{P}]$ has enough projectives and enough injectives. Moreover, each projective object is of the form $\mathcal{C}/[\mathcal{P}](-,X)$, and each injective object is of the form $\mathbb{E}(-,X)$.

\textup{(3)} The abelian category $\textup{mod-}(\mathcal{C}/[\mathcal{I}])^{\textup{op}}$ has enough projectives and enough injectives. Moreover, each projective object is of the form $\mathcal{C}/[\mathcal{I}](X,-)$, and each injective object is of the form $\mathbb{E}(X,-)$.
\end{thm}

\begin{proof}
(1) The first assertion follows from Theorem \ref{thm3.2}.
Assume that $$\delta=(\xymatrix{X_1\ar[r]^{f_1}&X_2\ar[r]^{f_2} & X_3\ar@{-->}[r]^{\rho}&})$$ is an $\mathfrak{s}$-triangle. Recall that $F(\delta)=\textup{Coker}(\mathcal{C}/[\mathcal{P}](-,f_2))$ and $\delta^\ast=\textup{Coker}(\mathcal{C}(-,f_2))$. Since $\delta^\ast(\mathcal{P})=0$, we can view $\delta^\ast$ as a finitely presented $\mathcal{C}/[\mathcal{P}]$-module. Thus $F(\delta)=\delta^\ast$. Similarly, we have $G(\delta)=\delta_\ast$.

(2) and (3) follows from (1),  Proposition \ref{prop3.2} and Example \ref{ex3.2}.
\end{proof}

\begin{cor}\label{prop3.2.1}
Let $\mathcal{C}$ be an extriangulated category with enough projectives $\mathcal{P}$ and enough injectives $\mathcal{I}$. Then there is a duality
$$\Phi:\textup{mod-}\mathcal{C}/[\mathcal{P}]\rightarrow\textup{mod-}(\mathcal{C}/[\mathcal{I}])^{\textup{op}},\ \ \ \delta^\ast\mapsto \delta_\ast.$$
Moreover, by restrictions, we obtain the following two dualities
$$\Phi:\textup{proj-}\mathcal{C}/[\mathcal{P}]\rightarrow\textup{inj-}(\mathcal{C}/[\mathcal{I}])^{\textup{op}}, \ \ \ \mathcal{C}/[\mathcal{P}](-,X)\mapsto \mathbb{E}(X,-).$$
$$\Phi:\textup{inj-}\mathcal{C}/[\mathcal{P}]\rightarrow\textup{proj-}(\mathcal{C}/[\mathcal{I}])^{\textup{op}}, \ \ \ \mathbb{E}(-,X)\mapsto\mathcal{C}/[\mathcal{I}](X,-).$$
\end{cor}

\begin{proof}
It is a direct consequence of Theorem \ref{rem3.2}.
\end{proof}

%The following result is implied in Proposition \ref{prop3.2.1}.

\begin{cor} (\cite[Proposition 4.9]{[INP]})
Let $\mathcal{C}$ be an extriangulated category with enough projectives $\mathcal{P}$ and enough injectives $\mathcal{I}$.

\textup{(1)} There is an isomorphism between $\mathcal{C}/[\mathcal{P}](Y,X)$ and the group of natural transformations from $\mathbb{E}(X,-)$ to $\mathbb{E}(Y,-)$.

\textup{(2)} There is an isomorphism between $\mathcal{C}/[\mathcal{I}](X,Y)$ and the group of natural transformations from $\mathbb{E}(-,X)$ to $\mathbb{E}(-,Y)$.
\end{cor}

Now we have the following Auslander-Reiten duality and defect formula for extriangulated categories.

\begin{prop}\label{thm4.4.2}
Let $\mathcal{C}$ be an Ext-finite extriangulated category with enough projectives $\mathcal{P}$ and enough injectives $\mathcal{I}$. Assume that either $\mathcal{C}/[\mathcal{P}]$ or $\mathcal{C}/[\mathcal{I}]$ is a dualizing $k$-variety. Then there is an equivalence $\tau:\mathcal{C}/[\mathcal{P}]\cong\mathcal{C}/[\mathcal{I}]$ satisfying the following properties:

\textup{(1)} $D\mathbb{E}(-,X)\cong\mathcal{C}/[\mathcal{P}](\tau^{-1} X,-)$, $D\mathbb{E}(X,-)\cong\mathcal{C}/[\mathcal{I}](-,\tau X)$.

\textup{(2)} $D\delta_\ast=\delta^\ast\tau^{-1}$, $D\delta^\ast=\delta_\ast\tau$ for each $\mathfrak{s}$-triangle $\delta$.
\end{prop}

\begin{proof} Without loss of generality, we assume that $\mathcal{C}/[\mathcal{I}]$ is a dualizing $k$-variety.
 The composition of  $\Phi:\textup{mod-}\mathcal{C}/[\mathcal{P}]\rightarrow \textup{mod-}(\mathcal{C}/[\mathcal{I}])^{\textup{op}}$ and  $D:\textup{mod-}(\mathcal{C}/[\mathcal{I}])^{\textup{op}}\rightarrow\textup{mod-}\mathcal{C}/[\mathcal{I}]$  defines an equivalence
$$\Theta: \ \textup{mod-}\mathcal{C}/[\mathcal{P}] \xrightarrow{\Phi}  \textup{mod-}(\mathcal{C}/[\mathcal{I}])^{\textup{op}} \xrightarrow{D}\textup{mod-}\mathcal{C}/[\mathcal{I}].$$
It follows that $\Theta(\mathcal{C}/[\mathcal{P}](-,X))=D\mathbb{E}(X,-)\cong\mathcal{C}/[\mathcal{I}](-,Y)$ for some $Y\in\mathcal{C}$.
Therefore, there is an equivalence $\tau: \mathcal{C}/[\mathcal{P}]\cong\mathcal{C}/[\mathcal{I}]$ mapping $X$ to $Y$. The equivalence $\tau$ induces an equivalence $\tau^{-1}_\ast: \textup{mod-}\mathcal{C}/[\mathcal{P}]\cong \textup{mod-}\mathcal{C}/[\mathcal{I}], F\mapsto F\tau^{-1}$, such that $D\Phi=\tau^{-1}_\ast$. Assume that $\delta$ is an $\mathfrak{s}$-triangle, then $D\Phi(\delta^\ast)=D\delta_\ast$. On the other hand, $\tau^{-1}_\ast(\delta^\ast)=\delta^\ast\tau^{-1}$. Hence, we have $D\delta_\ast=\delta^\ast\tau^{-1}$. It follows that $D\delta^*=\delta_*\tau$. If $\delta=I_{X}$, then by Example \ref{ex3.2} we have $\delta^*=\mathbb{E}(-,X)$ and $\delta_*=\mathcal{C}/\mathcal{I}(X,-)$. Therefore, we have
$D\mathbb{E}(-,X)\cong \mathcal{C}/[\mathcal{I}](X,\tau-)\cong\mathcal{C}/[\mathcal{P}](\tau^{-1}X,-)$ since $D\delta^*=\delta_*\tau$.
\end{proof}

\section{Application to rigid subcategories}
In this section, we will investigate the second application of Theorem \ref{3.1} in the case when $\mathcal{M}$ is a rigid subcategory of  $\mathcal{C}$, that is, $\mathbb{E}(M,M')=0$ for any objects $M,M'\in\mathcal{M}$.

Let $\mathcal{C}'$ and $\mathcal{C}''$ be two full subcategories of $\mathcal{C}$. We denote by Cocone$(\mathcal{C}',\mathcal{C}'')$ the full  subcategory of $\mathcal{C}$ of objects $X$ admitting an $\mathfrak{s}$-triangle $\xymatrixrowsep{0.1pc}\xymatrix{X\ar[r]&C'\ar[r] & C''\ar@{-->}[r]&}$ with $C'\in\mathcal{C}'$ and $C''\in\mathcal{C}''$. We denote by Cone$(\mathcal{C}',\mathcal{C}'')$ the full subcategory of objects $X$ admitting an $\mathfrak{s}$-triangle $\xymatrixrowsep{0.1pc}\xymatrix{C'\ar[r]&C''\ar[r] & X\ar@{-->}[r]&}$ with $C'\in\mathcal{C}'$ and $C''\in\mathcal{C}''$.

 For convenience we let $\mathcal{M}_L=\textup{Cocone}(\mathcal{M},\mathcal{M})$ and $\mathcal{M}_R=\textup{Cone}(\mathcal{M},\mathcal{M})$. If $\mathcal{C}$ has enough projectives $\mathcal{P}$, then we let $\Omega\mathcal{M}$=Cocone$(\mathcal{P},\mathcal{M})$. If $\mathcal{C}$ has enough injectives $\mathcal{I}$, then we let $\Sigma\mathcal{M}$=Cone$(\mathcal{M},\mathcal{I})$. %We say $\mathcal{M}$ is {\em rigid} if $\mathbb{E}(M,M')=0$ for any objects $M,M'\in\mathcal{M}$.

\begin{lem}\label{rem2.1}
Let $\mathcal{C}$ be an extriangulated category and $\mathcal{M}$ be a rigid subcategory of $\mathcal{C}$. If $\xymatrixrowsep{0.1pc}\xymatrix{X\ar[r]^{k}&M_1\ar[r]^{f} & M_2\ar@{-->}[r]^{\delta}&}$ is an $\mathfrak{s}$-triangle with $M_i\in\mathcal{M}$, then $k$ is a left $\mathcal{M}$-approximation of $X$.
\end{lem}

\begin{proof}
For any $M\in\mathcal{M}$, by Lemma \ref{lem2.0}
we have the following exact sequence $$\mathcal{C}(M_2,M)\rightarrow\mathcal{C}(M_1,M)\rightarrow\mathcal{C}(X,M)\rightarrow \mathbb{E}(M_2,M)=0.$$
Hence, $k$ is a left $\mathcal{M}$-approximation of $X$.
\end{proof}

%From now on, we assume that $\mathcal{M}$ is a rigid subcategory of $\mathcal{C}$ containing $\mathcal{P}$.

\begin{lem}\label{lem2.2} Let $\mathcal{C}$ be an extriangulated category with enough projectives $\mathcal{P}$ and $\mathcal{M}$ be a rigid subcategory of $\mathcal{C}$ containing $\mathcal{P}$.
Assume that the following diagram
$$\xymatrix{
 X\ar[r]^{k}\ar[d]^{g} & M_1 \ar[r]^{f}\ar[d]^{a} & M_2\ar@{-->}[r]^{\delta}\ar[d]^{b} & \\
 X'\ar[r]^{k'} & M_1' \ar[r]^{f'} &  M_2' \ar@{-->}[r]^{\delta'}&
}$$
is a morphism of $\mathfrak{s}$-triangles with $M_i, M_i'\in\mathcal{M}$. Then

\textup{(1)} The following statements are equivalent.

\ \ \ \textup{(a)}  The morphism $b$ factors through $f'$.

\ \ \ \textup{(b)}  The morphism $(a,b)$ factors through some object in $\textup{s-epi}(\mathcal{M})$.

\ \ \ \textup{(c)}  The morphism $g$ factors through some object in $\mathcal{M}$.

\textup{(2)} The following statements are equivalent.

\ \ \ \textup{(a)} The morphism $\underline{a}$ factors through $\underline{f}$ in $\mathcal{M}/[\mathcal{P}]$.

\ \ \ \textup{(b)} The morphism $(a,b)$ factors through some object in $\textup{sp-epi}(\mathcal{M})$.

\ \ \ \textup{(c)} The morphism $g$ factors through some object in $\Omega\mathcal{M}$.
\end{lem}

\begin{proof}
 (1) We note that (a)$\Leftrightarrow$(b) follows from Lemma \ref{lem3.0}(1).

 (a)$\Rightarrow$(c).  Assume that  $b$ factors through $f'$. It follows that  $g$ factors through $k$ by Lemma \ref{2.1}, which implies that $g$ factors through $M_1\in\mathcal{M}$.

 (c)$\Rightarrow$(a). Suppose that $g$ has a factorization $X\xrightarrow{g_1} M\xrightarrow{g_2} X'$ with $M\in\mathcal{M}$. By Lemma \ref{rem2.1} we have $g_1$ factors through $k$. Thus $g$ factors through $k$. It follows that $b$ factors through $f'$ by Lemma \ref{2.1}.

 (2) We note that (a)$\Leftrightarrow$(b) follows from Lemma \ref{lem3.0}(2).

(a)$\Rightarrow$(c). Suppose that there is a morphism $\underline{p}:M_2\rightarrow M_1'$ such that $\underline{p}\underline{f}=\underline{a}$.
Since $\mathcal{C}$ has enough projectives, there is an $\mathfrak{s}$-triangle
 $\xymatrix{\Omega M_1'\ar[r]^{a_1'}& P\ar[r]^{a_1} & M_1'\ar@{-->}[r]^{\delta''}&\\
 }$ %deflation $a_1:P\rightarrow M_1'$
 with $P\in\mathcal{P}$. It is easy to see that  $pf-a$ factors through $a_1$. Assume that $pf-a=a_1a_2$ where $a_2:M_1\rightarrow P$. By (ET4)$^{\textup{op}}$, we have the following diagram
$$\xymatrix{
\Omega M'_1\ar[r]^{d'}\ar@{=}[d] & Y \ar[r]^{d}\ar[d]^{c} & X'\ar@{-->}[r]^{k'^*\delta''}\ar[d]^{k'} & \\
\Omega M'_1\ar[r]^{a_1'} & P \ar[r]^{a_1}\ar[d]^{c'} &  M_1' \ar@{-->}[r]^{\delta''}\ar[d]^{f'} & \\
& M_2'\ar@{=}[r]\ar@{-->}[d]^{\rho} & M_2'\ar@{-->}[d]^{\delta'} \\
& & & \\
 }$$
where  the first row and the second column are $\mathfrak{s}$-triangles.
It follows that $Y\in\Omega\mathcal{M}$. Since the upper-right square of the above diagram obtained by (ET4)$^{\textup{op}}$ is a weak pullback and  $(k',a_1)\left(
                                                           \begin{smallmatrix}
                                                             g \\
                                                             a_2k \\
                                                           \end{smallmatrix}
                                                         \right)=ak+a_1a_2k=pfk=0$, there exists a morphism $h:X\rightarrow Y$ such that $dh=g$ and $ch=a_2k$.  Therefore,  $g$ factors through $Y\in\Omega\mathcal{M}$.

(c)$\Rightarrow$(a). Suppose that $g$ has a factorization $X\xrightarrow{g_1}\Omega M\xrightarrow{g_2} X'$ with $\Omega M\in \Omega\mathcal{M}$. Then by Lemma \ref{rem2.1} and (ET3) we complete the following morphism of $\mathfrak{s}$-triangles
$$\xymatrix{
 X\ar[r]^{k}\ar[d]^{g_1} & M_1 \ar[r]^{f}\ar@{-->}[d]^{a_1} & M_2\ar@{-->}[r]^{\delta}\ar@{-->}[d]^{b_1} & \\
 \Omega M \ar[r]^{i}\ar[d]^{g_2} & P \ar[r]^{\pi}\ar@{-->}[d]^{a_2} & M\ar@{-->}[r]^{\delta''}\ar@{-->}[d]^{b_2}& \\
 X'\ar[r]^{k'} & M_1' \ar[r]^{f'} &  M_2' \ar@{-->}[r]^{\delta'}& \\
}$$ where $P\in\mathcal{P}$. Since $(a-a_2a_1)k=k'(g-g_2g_1)=0$, there exists a morphism $p:M_2\rightarrow M_1'$ such that $a-a_2a_1=pf$. Therefore, $\underline{a}=\underline{p}\underline{f}$.
\end{proof}

\begin{lem}\label{lem5.2}
Let $\mathcal{C}$ be an extriangulated category with enough projectives $\mathcal{P}$ and $\mathcal{M}$ be a rigid subcategory of $\mathcal{C}$ containing $\mathcal{P}$. Then

\textup{(1)} $\mathfrak{s}\textup{-def}(\mathcal{M})/[\textup{s-epi}(\mathcal{M})]\cong\mathcal{M}_{L}/[\mathcal{M}]$.

\textup{(2)} $\mathfrak{s}\textup{-def}(\mathcal{M})/[\textup{sp-epi}(\mathcal{M})]\cong\mathcal{M}_{L}/[\Omega\mathcal{M}]$.
\end{lem}

\begin{proof}
(1) For any object $f:M_1\rightarrow M_2$ in $\mathfrak{s}\textup{-def}(\mathcal{M})$,
there exists an $\mathfrak{s}$-triangle $\xymatrix{X\ar[r]^{k}&M_1\ar[r]^{f} & M_2\ar@{-->}[r]^{\delta}&}$ where $X$ is unique under isomorphism.
By (ET3)$^{\textup{op}}$,  the following commutative diagram
$$\xymatrix{
 X\ar[r]^{k}\ar@{-->}[d]^{g} & M_1 \ar[r]^{f}\ar[d]^{a} & M_2\ar@{-->}[r]^{\delta}\ar[d]^{b} & \\
 X'\ar[r]^{k'} & M_1' \ar[r]^{f'} &  M_2' \ar@{-->}[r]^{\delta'}&
}$$
whose rows are $\mathfrak{s}$-triangles can be completed to a morphism of  $\mathfrak{s}$-triangles. %is a morphism of $\mathfrak{s}$-triangles with $M_i, M_i'\in\mathcal{M}$.
The morphism $g:X\rightarrow X'$ is not unique in general. Assume that $g':X\rightarrow X'$ is another morphism such that $(g',a,b)$ is a morphism of $\mathfrak{s}$-triangles. Then $(g-g',0,0)$ is also a morphism of $\mathfrak{s}$-triangles. Lemma \ref{2.1} implies that $g-g'$ factors through $k$, that is, $g-g'$ factors through $M_1\in\mathcal{M}$. It follows that $\underline{g}=\underline{g}'$ in $\mathcal{M}_L/[\mathcal{M}]$.

 Hence the assignment $(M_{1}\xrightarrow{f} M_2)\mapsto X, (a,b)\mapsto \underline{g}$ defines a well-defined functor $F: \mathfrak{s}\textup{-def}(\mathcal{M})\rightarrow\mathcal{M}_L/[\mathcal{M}]$. It is clear that $F$ is dense. The functor $F$ is full by Lemma \ref{rem2.1} and (ET3). Lemma \ref{lem2.2}(1) implies that $F$ induces an equivalence $\mathfrak{s}\textup{-def}(\mathcal{M})/\mathcal{R}\cong\mathcal{M}_L/[\mathcal{M}]$, that is, $\mathfrak{s}\textup{-def}(\mathcal{M})/[\textup{s-epi}(\mathcal{M})]\cong\mathcal{M}_{L}/[\mathcal{M}]$.

(2) For any $X\in \mathcal{M}_{L}$, we fix an $\mathfrak{s}$-triangle $\xymatrix{X\ar[r]^{k}&M_1\ar[r]^{f} & M_2\ar@{-->}[r]^{\delta}&}$ with $M_1,M_2\in\mathcal{M}$.
Assume that $g:X\rightarrow X'$ is a morphism in $\mathcal{M}_{L}$. Then by Lemma \ref{rem2.1} and (ET3) there exist two morphisms $a: M_1\rightarrow M_1'$ and $b:M_2\rightarrow M_2'$ such that the following is a morphism of $\mathfrak{s}$-triangles
$$\xymatrix{
 X\ar[r]^{k}\ar[d]^{g} & M_1 \ar[r]^{f}\ar[d]^{a} & M_2\ar@{-->}[r]^{\delta}\ar[d]^{b} & \\
 X'\ar[r]^{k'} & M_1' \ar[r]^{f'} &  M_2' \ar@{-->}[r]^{\delta'}&
}$$
where $M'_1, M_2'\in\mathcal{M}$. Suppose that the morphisms $a': M_1\rightarrow M_1'$ and $b':M_2\rightarrow M_2'$ satisfying $(g,a',b')$ is also a morphism of $\mathfrak{s}$-triangles. Then $(0,a-a',b-b')$ is a morphism of $\mathfrak{s}$-triangles. It follows that $(a-a',b-b')$ factors through some object in $\textup{sp-epi}(\mathcal{M})$ by Lemma \ref{lem2.2}(2). We have $\underline{(a,b)}=\underline{(a',b')}$ in $\mathfrak{s}\textup{-def}(\mathcal{M})/[\textup{sp-epi}(\mathcal{M})]$.
Therefore, the assignment $X\mapsto (M_{1}\xrightarrow{f} M_2), g\mapsto \underline{(a,b)}$ defines a well-defined functor $G: \mathcal{M}_L\rightarrow \mathfrak{s}\textup{-def}(\mathcal{M})/[\textup{sp-epi}(\mathcal{M})]$.

It is easy to see that $G$ is full and dense. Lemma \ref{lem2.2}(2) implies that $G$ induces an equivalence $\mathcal{M}_{L}/[\Omega\mathcal{M}]\cong\mathfrak{s}\textup{-def}(\mathcal{M})/[\textup{sp-epi}(\mathcal{M})]$.%For each morphism $\underline{f}:M_1\rightarrow M_2$ in $\textup{Mor}(\mathcal{M}/[\mathcal{P}])$, we assume that $\pi:P\rightarrow M_2$ is an $\mathfrak{s}$-deflation with $P\in\mathcal{P}$. It follows from Lemma \ref{lem2.4} that $(f,-\pi): X\oplus P\rightarrow Y$ is an $\mathfrak{s}$-deflation and $\underline{(f,-\pi)}\cong \underline{f}$. Thus the functor $G$ is dense.
%Assume that $\xymatrix{X\ar[r]^{k}&M_1\ar[r]^{f} & M_2\ar@{-->}[r]^{\delta}&}$ and $\xymatrix{X'\ar[r]^{k'}&M'_1\ar[r]^{f'} & M'_2\ar@{-->}[r]^{\delta'}&}$ are two $\mathfrak{s}$-triangles with $M_i,M'_i\in\mathcal{M}$. Suppose that $(\underline{a},\underline{b})$ is a morphism in $\textup{Mor}(\mathcal{M}/[\mathcal{P}])$ from $\underline{f}$ to $\underline{f'}$.
%Then $\underline{b}\underline{f}=\underline{f'}\underline{a}$. By Lemma \ref{lem2.4}, there exists an object $P\in\mathcal{P}$ and two morphisms $u:M_1\rightarrow P$ and $v:P\rightarrow M_1'$ such that $f'(a-vu)=bf$. Since $f'(a-vu)k=bfk=0$, there exists a morphism $g:X\rightarrow X'$ such that $(a-vu)k=k'g$. It follows that $G(g)=(\underline{a},\underline{b})$. Thus the functor $G$ is full.
%Lemma \ref{lem2.2}(2) implies that $F$ induces an equivalence $\mathcal{M}_{L}/[\Omega\mathcal{M}]\cong\textup{Mor}(\mathcal{M}/[\mathcal{P}])/\mathcal{R}'$. Therefore,  we obtain  $\mathcal{M}_{L}/[\Omega\mathcal{M}]\cong(\textup{mod-}(\mathcal{M}/[\mathcal{P}])^{\textup{op}})^{\textup{op}}$ by Lemma \ref{lem2.1}(2).
\end{proof}

\begin{thm}\label{thm4.1}
Let $\mathcal{C}$ be an extriangulated category and $\mathcal{M}$ be a rigid subcategory of $\mathcal{C}$.

\textup{(1)} If $\mathcal{C}$ has enough projectives $\mathcal{P}$ and $\mathcal{M}$ contains $\mathcal{P}$, then $\mathcal{M}_{L}/[\mathcal{M}]\cong\textup{mod-}(\mathcal{M}/[\mathcal{P}])$ and  $\mathcal{M}_{L}/[\Omega\mathcal{M}]\cong(\textup{mod-}(\mathcal{M}/[\mathcal{P}])^{\textup{op}})^{\textup{op}}$.

\textup{(2)} If $\mathcal{C}$ has enough injectives $\mathcal{I}$ and $\mathcal{M}$  contains $\mathcal{I}$, then $\mathcal{M}_{R}/[\mathcal{M}]\cong(\textup{mod-}(\mathcal{M}/[\mathcal{I}])^{\textup{op}})^{\textup{op}}$ and  $\mathcal{M}_{R}/[\Sigma\mathcal{M}]\cong\textup{mod-}(\mathcal{M}/[\mathcal{I}])$.
\end{thm}

\begin{proof} We only prove (1). It follows from Theorem \ref{thm3.2} and Lemma \ref{lem5.2}.
\end{proof}

\begin{defn}(\cite[Definition 5.3]{[LN]})
Let $\mathcal{C}$ be an extriangulated category with enough projectives and enough injectives. A full subcategory $\mathcal{M}$ is called {\em $n$-cluster tilting} for some integer $n\geq 2$, if it satisfies the following conditions.

(1) $\mathcal{M}$ is functorially finite in $\mathcal{C}$.

(2) $X\in\mathcal{M}$ if and only if $\mathbb{E}^i(X,\mathcal{M})=0$ for $i\in\{1,2,\cdots,n-1\}$.

(3) $X\in\mathcal{M}$ if and only if $\mathbb{E}^i(\mathcal{M},X)=0$ for $i\in\{1,2,\cdots,n-1\}$.
\end{defn}

In particular, a 2-cluster tilting subcategory of $\mathcal{C}$ is simply called {\em cluster tilting subcategory}.

Assume that $\mathcal{M}$ is an $n$-cluster tilting subcategory. Define $$^{\bot_{n-2}}\mathcal{M}=\{X\in\mathcal{C}|\mathbb{E}^i(X,\mathcal{M})=0,i\in\{1,2,\cdots,n-2\}\}$$ $$\textup{and}\ \  \mathcal{M}^{\bot_{n-2}}=\{X\in\mathcal{C}|\mathbb{E}^i(\mathcal{M},X)=0,i\in\{1,2,\cdots,n-2\}\}.$$

\begin{prop}
Let $\mathcal{M}$ be an $n$-cluster tilting subcategory of an extriangulated category $\mathcal{C}$, then $\mathcal{M}_L=^{\bot_{n-2}}\mathcal{M}$ and  $\mathcal{M}_R=\mathcal{M}^{\bot_{n-2}}$.
\end{prop}

\begin{proof}
We only prove the first equation. Let $X\in^{\bot_{n-2}}\mathcal{M}$ and $\xymatrix{X\ar[r]^{x}&I\ar[r]^{y} & C\ar@{-->}[r]^{\delta}&}$ be an $\mathfrak{s}$-triangle with $I\in\mathcal{I}$. Assume that $f:X\rightarrow M$ is a left $\mathcal{M}$-approximation of $X$. By Lemma \ref{2.2}, we have the following morphism of  $\mathfrak{s}$-triangles
$$\xymatrix{
 X\ar[r]^{x}\ar[d]^{f} & I \ar[r]^{y}\ar[d]^{g} & C\ar@{-->}[r]^{\delta}\ar@{=}[d] & \\
 M\ar[r]^{x'} & C' \ar[r]^{y'} & C \ar@{-->}[r]^{f_*\delta}&
}$$ such that $\xymatrix{X\ar[r]^{\left(
                                    \begin{smallmatrix}
                                      f \\
                                      x \\
                                    \end{smallmatrix}
                                  \right)
}&M\oplus I\ar[r]^{(x',-g)} & C'\ar@{-->}[r]^{y'^*\delta}&}$ is an $\mathfrak{s}$-triangle. We claim that $C'\in \mathcal{M}$, thus $X\in \mathcal{M}_L$. In fact, for each $M'\in\mathcal{M}$, by Lemma \ref{lem2.5} we have the following exact sequence
$$\mathcal{C}(C',M')\rightarrow\mathcal{C}(M\oplus I,M')\rightarrow\mathcal{C}(X,M')\rightarrow\mathbb{E}(C',M')\rightarrow\mathbb{E}(M\oplus I,M')\rightarrow\mathbb{E}(X,M')$$
$$\rightarrow\cdots\rightarrow\mathbb{E}^i(M\oplus I,M')\rightarrow\mathbb{E}^i(X,M')\rightarrow\mathbb{E}^{i+1}(C',M')\rightarrow
\mathbb{E}^{i+1}(M\oplus I,M')\rightarrow\cdots.$$
Since $\mathbb{E}^{i}(M\oplus I,M')=\mathbb{E}^{i+1}(M\oplus I,M')=0$, we get $\mathbb{E}^i(X,M')\cong\mathbb{E}^{i+1}(C',M')$ for $i\in\{1,2,\cdots,n-2\}$. Thus, as $X\in^{\bot_{n-2}}\mathcal{M}$,  $\mathbb{E}^{i}(C',M')=0$ if $2\leq i\leq n-1$. Noting that $f:X\rightarrow M$ is a left $\mathcal{M}$-approximation, we have $\mathbb{E}(C',M')=0$. Therefore $C'\in\mathcal{M}$.

Suppose that $X\in\mathcal{M}_L$. There exists an $\mathfrak{s}$-triangle $\xymatrix{X\ar[r]^{f_1}&M_1\ar[r]^{f_2} & M_2\ar@{-->}[r]^{\delta}&}$ with $M_1, M_2\in\mathcal{M}$. For any $M\in\mathcal{M}$, we have an exact sequence $$0=\mathbb{E}^i(M_1,M)\rightarrow\mathbb{E}^i(X,M)\rightarrow\mathbb{E}^{i+1}(M_2,M)=0$$ for $i\in\{1,2,\cdots,n-2\}$. Therefore, $X\in^{\bot_{n-2}}\mathcal{M}$.
\end{proof}

\begin{cor}\label{cor5.1}
Let $\mathcal{C}$ be an extriangulated category with enough projectives $\mathcal{P}$ and enough  injectives $\mathcal{I}$. If $\mathcal{M}$ is a cluster tilting subcategory of $\mathcal{C}$, then

\textup{(1)} $\mathcal{C}/[\mathcal{M}]\cong\textup{mod-}(\mathcal{M}/[\mathcal{P}])\cong(\textup{mod-}(\mathcal{M}/[\mathcal{I}])^{\textup{op}})^{\textup{op}}$.

\textup{(2)} $\mathcal{C}/[\Omega\mathcal{M}]\cong(\textup{mod-}(\mathcal{M}/[\mathcal{P}])^{\textup{op}})^{\textup{op}}$.

\textup{(3)} $\mathcal{C}/[\Sigma\mathcal{M}]\cong\textup{mod-}(\mathcal{M}/[\mathcal{I}])$.
\end{cor}

\begin{rem}
Let $\mathcal{M}$ be a cluster tilting subcategory of an extriangulated category $\mathcal{C}$. Then the quotient category $\mathcal{C}/[\mathcal{M}]$ is abelian can follows from the theory of cotorsion pairs on extriangulated categories or that of one-sided triangulated categories; see \cite[Theorem 3.2]{[LN]} and \cite[Theorem 3.3]{[HZ]}.
\end{rem}

\end{document}